\documentclass[12pt,letterpaper]{amsart}

\usepackage[none]{hyphenat}
\usepackage[utf8]{inputenc}
\usepackage[T1]{fontenc}
\usepackage{amsfonts}
\usepackage[leqno]{amsmath}
\usepackage{amssymb, comment, float}
\usepackage[dvipsnames]{xcolor}
\usepackage[foot]{amsaddr}
\usepackage[shortlabels]{enumitem}
\usepackage{mathdots}
\usepackage{todonotes}
\usepackage[footskip=0.5in, headheight = 0.5in, top=1.25in, bottom=1.25in, right=1in, left=1in]{geometry}
\usepackage{hyperref}
\usepackage[center]{caption}
\usepackage{eufrak}
\usepackage{marvosym}
\usepackage{latexsym}
\usepackage{tikz}
\usepackage{subfig}
\usepackage{thm-restate}

\hypersetup{
colorlinks,
linkcolor=blue,
urlcolor=black,
citecolor=blue,}

\newtheorem{theorem}{Theorem}[section]
\newtheorem{lemma}[theorem]{Lemma}

\DeclareMathOperator{\tw}{tw}

\DeclareMathOperator{\ta}{{\text{\sf{tree}}-\alpha}}

\newcommand{\mca}{\mathcal}
\newcommand{\poi}{\mathbb N}

\newcommand{\blackref}[1]{\hypersetup{linkcolor=black}\ref{#1}\hypersetup{linkcolor=blue}}

\newcounter{tbox}

\newcommand{\mylongtitle}[1]{%
	\ifodd\value{page}%
		\protect\parbox{0.97\linewidth}{#1}\hfill%
	\else%
		\hfill\protect\parbox{0.97\linewidth}{#1}%
	\fi%
}

\title[Tree-independence number VIII.]{Tree-independence number VIII. 3PC-free graphs have logarithmic tree-independence number}

\author{Maria Chudnovsky$^\dagger$}
\author{Julien Codsi$^\ddagger$}
\author{Ajaykrishnan E S$^\star$}
\author{Daniel Lokshtanov$^\star$}
\author{Ilya Maier$^\S$}

\thanks{$^\dagger$ Princeton University, Princeton, NJ, USA.
	Supported by NSF Grant DMS-2348219, NSF Grant CCF-2505100, AFOSR grant FA9550-25-1-0275 and a Guggenheim Fellowship}
\thanks{$^\ddagger$ Princeton University, Princeton, NJ, USA. Supported by NSF Grant DMS-2348219, NSF Grant CCF-2505100, AFOSR grant FA9550-25-1-0275}
\thanks{$^\star$ Department of Computer Science, University of California Santa Barbara, Santa Barbara, CA, USA. Supported by NSF Grant CCF-2505099}
\thanks{$^\S$ Princeton University, Princeton, NJ, USA}

\date{\today}
\begin{document}
\raggedbottom
\maketitle

\begin{abstract}
	
	A {\it three-path-configuration} is a graph consisting of three pairwise internally vertex-disjoint paths, the union of every two of which is an induced cycle of length at least four.
	A graph is {\it 3PC-free} if no induced subgraph of it is a three-path-configuration.
	We show that 3PC-free graphs have logarithmic tree-independence number, which is asymptotically best possible.
	
\end{abstract}

\section{Introduction}

All graphs in this paper are finite and simple and all logarithms are base $2$.
We treat a graph and its vertex set interchangeably.
Throughout this paper we use $n$ to denote the number of vertices of a graph $G$.
For graphs $G$ and $H$, we say that $G$ is {\em $H$-free} if no induced subgraph of $G$ is isomorphic to $H$.
For $X,Y\subseteq V(G)$, we say that $X$ and $Y$ are {\em anticomplete} if $X\cap Y=\emptyset$ and there is no edge in $G$ with an end in $X$ and an end in $Y$.
The {\em closed neighborhood} of a set $X$, denoted by $N[X]$, is the set of vertices of $G$ that are in $X$ or have at least one neighbor in $X$.
For standard graph-theoretic terminology, the reader is referred to \cite{Diestel}.

For a graph $G$, a {\em tree decomposition} $(T, \chi)$ of $G$ consists of a tree $T$ and a map $\chi: V(T) \rightarrow 2^{V(G)}$ with the following properties:
\begin{enumerate}[\rm (i)]
	\item For every $v \in V(G)$, there exists $t \in V(T)$ such that $v \in \chi(t)$.
	\item For every $uv\in E(G)$, there exists $t \in V(T)$ such that $u,v \in \chi(t)$.
	\item For every $v \in V(G)$, the subgraph of $T$ induced by $\{t \in V(T) \mid v \in \chi(t)\}$ is connected.
\end{enumerate}
The {\em width} of a tree decomposition $(T, \chi)$ is $\max_{t \in V(T)}|\chi(t)|-1$.
The {\em independence number} of a tree decomposition $(T, \chi)$ is $\max_{t \in V(T)} \alpha(\chi(t))$.
The {\em treewidth} of $G$, denoted by $\tw(G)$, is the minimum width of a tree decomposition of $G$.
The {\em tree-independence number} of $G$, denoted by $\ta(G)$, is the minimum independence number of a tree decomposition of $G$.

A {\em path} in a graph is an induced subgraph that is a path.
The {\em interior} of a path $P$ is denoted by $P^*$.
A {\em theta} is a graph consisting of two non-adjacent vertices $a,b$ and three paths $P_1, P_2, P_3$ from $a$ to $b$, such that $P_i^*$ is anticomplete to $P_j^*$ for $i\neq j$.
A {\it three-path-configuration} is a graph consisting of three pairwise internally vertex-disjoint paths, the union of every two of which is an induced cycle of length at least four.
A graph is {\it 3PC-free} if no induced subgraph of it is a three-path-configuration.

It is known that 3PC-free graphs have polylogarithmic tree-independence number and if their clique number is bounded by a constant $t$, then they also have logarithmic treewidth.

\begin{theorem}[Abrishami, Chudnovsky, Hajebi, Spirkl \cite{tw3}]
\label{thm:3pc-tw}

	For every $t\in \poi$ there exists a constant $c_t$ such that if a graph $G$ is 3PC-free and $K_t$-free, then $\tw(G)\leq c_t \log n$.

\end{theorem}

\begin{theorem}[Chudnovsky, Hajebi, Lokshtanov, Spirkl \cite{ta2}]
\label{thm:3pc-ta}

	There exists a constant $c$ such that if a graph $G$ with $n>1$ is 3PC-free, then $\ta(G) \le c (\log n)^2$.

\end{theorem}

On the other hand, the best lower bound on the tree-independence number of 3PC-free graphs is $\log n$ (Sintiari and Trotignon \cite{triangle-free-layered-wheels}).
This led to a question \cite{3pc-tree-alpha-question,ta2} of whether 3PC-free graphs with superlogarithmic tree-independence number exist.

Recently, Hajebi and Spirkl \cite{2dbs} introduced a powerful new notion of loose separability.
A direct application of their method allows us to close the gap between the upper and the lower bounds.
We prove:

\begin{theorem}
\label{thm:main}
	
	There exist constants $c,d$ such that if a graph $G$ with $n>1$ is 3PC-free, then $\ta(G) \le c + d\log n$.
	
\end{theorem}

Note that since the class of theta-free graphs is ``$\chi$-bounded'' (see \cite{chi-survey} for details), Theorem \ref{thm:main} yields another proof of Theorem \ref{thm:3pc-tw}.

\section{Proof}

We include some definitions from \cite{ta2} and \cite{2dbs} that will be useful for us.
If $S$ is a set of sets, a \textit{hitting set for $S$} is a set $X$ such that $X\cap Y\neq\varnothing$ for every $Y\in S$.
For $b\in \mathbb N$, a \textit{$b$-system} $\mca S$ in $G$ is a set of non-empty subsets of $V(G)$, each of cardinality at most $b$ (thus, $\mca S=\varnothing$ is allowed, whereas $\mca S=\{\varnothing\}$ is not).
We say that a $b$-system $\mca{S}$ in $G$ is \textit{anticomplete} if every two distinct sets $S_1,S_2\in \mca{S}$ are anticomplete in $G$.
We denote the set of all $(X,Y)$-paths in $G$ by $\mca{P}_G(X,Y)$.

For $b,c,d\in \mathbb N$, a graph $G$ is \textit{$(b,c)$-loosely $d$-$\alpha$-separable} if for every $b$-system $\mca{S}$ in $G$, either
\begin{itemize}
	\item there is a hitting set $X\subseteq V(G)$ for $\mca{S}$ with $\alpha(X)\leq c$; or
	\item there exist $S_1,S_2\in \mca{S}$ such that $\mca{P}_G(S_1,S_2)$ has a hitting set $Y\subseteq V(G)\setminus (S_1\cup S_2)$ with $\alpha(Y)\leq d$; in particular, $S_1$ and $S_2$ are anticomplete in $G$.
\end{itemize}

For a graph $G$, a function $w: V(G) \to [0,1]$ is a {\em normal weight function} on $G$ if $w(V(G))=1$.
Let $c \in [0, 1]$ and let $w$ be a normal weight function on $G$.
A set $X \subseteq V(G)$ is a {\em $(w,c)$-balanced separator} if $w(D) \leq c$ for every component $D$ of $G \setminus X$.
The set $X$ is a {\em $w$-balanced separator} if $X$ is a $(w,\frac{1}{2})$-balanced separator.
For $b\in \mathbb N$, a graph $G$ is \textit{$b$-separator-dominated} if for every normal weight function $w$ on $G$, there exists $X\subseteq V(G)$ with $|X|\le b$ such that $N[X]$ is a $w$-balanced separator in $G$.

Hajebi and Spirkl \cite{2dbs} prove:

\begin{lemma}[Hajebi, Spirkl \cite{2dbs}]
\label{lem:hit-vs-anti}

	For all $a,b,q\in \mathbb N$, there is a constant $c_{\blackref{lem:hit-vs-anti}}=c_{\blackref{lem:hit-vs-anti}}(a,b,q)\in \mathbb N$ such that for every $K_{a,a}$-free graph $G$ and every $b$-system $\mca{S}$ in $G$, either there is a hitting set $X\subseteq V(G)$ for $\mca{S}$ with $\alpha(X)< c_{\blackref{lem:hit-vs-anti}}$, or there is an anticomplete $b$-system $\mca{Q}\subseteq\mca{S}$ in $G$ with $|\mca{Q}|\geq q$.

\end{lemma}

\begin{lemma}[Hajebi, Spirkl \cite{2dbs}]
\label{lem:ta-bounded}

	Let $b,c,d\in \mathbb N$ and let $G$ be a $(b,c)$-loosely $d$-$\alpha$-separable graph such that every induced subgraph of $G$ is $b$-separator-dominated.
	Then, $\ta(G)\leq 10c+5d$.

\end{lemma}

The following was proved in \cite{ta2}:

\begin{theorem}[Chudnovsky, Hajebi, Lokshtanov, Spirkl \cite{ta2}]
\label{thm:dbs}

	There is an integer $b$ with the following property.
	Let $G$ be a 3PC-free graph, and let $w$ be a normal weight function on $G$.
	Then, there exists $Y \subseteq V(G)$ such that
	\begin{itemize}
		\item $|Y| \leq b$, and
		\item $N[Y]$ is a $w$-balanced separator in $G$.
	\end{itemize}

\end{theorem}

\begin{theorem}[Chudnovsky, Hajebi, Lokshtanov, Spirkl \cite{ta2}]
\label{thm:sep}

	Let $G$ be a 3PC-free graph and let $a,b \in V(G)$ be non-adjacent.
	Then, there is a hitting set $X \subseteq V(G) \setminus \{a,b\}$ for $\mca{P}_G(\{a\},\{b\})$ with $\alpha(X) \leq 32 \log n$.
	
\end{theorem}

\begin{proof}[Proof of Theorem \ref{thm:main}:]

	\

	By Theorem \ref{thm:dbs}, there exists an integer $b$ such that every 3PC-free graph is $b$-separator-{\allowbreak}dominated.
	In particular, every induced subgraph of $G$ is $b$-separator-{\allowbreak}dominated.
	Let $\mca S$ be a $b$-system in $G$.
	Since $G$ is theta-free, $G$ is $K_{3,3}$-free.
	By Lemma \ref{lem:hit-vs-anti}, there exists an integer $c$ such that either there is a hitting set $X\subseteq V(G)$ for $\mca S$ with $\alpha(G[X]) < c$, or there are two anticomplete sets $S_1,S_2\in \mca S$.
	Let $s_1\in S_1, s_2\in S_2$, and let $G_{s_1s_2}=(G\setminus (S_1\cup S_2))\cup \{s_1,s_2\}$.
	Theorem \ref{thm:sep} implies that there exists a hitting set $Y_{s_1s_2}$ for $\mca{P}_{G_{s_1s_2}}(\{s_1\},\{s_2\})$ in $G_{s_1s_2}$ with $\alpha(Y_{s_1s_2})\le 32\log n$.
	Thus, $\mca{P}_G(S_1,S_2)$ has a hitting set $Y=\cup_{s_1\in S_1,s_2\in S_2}Y_{s_1s_2}\subseteq V(G)\setminus (S_1\cup S_2)$ with $\alpha(G[Y])\leq 32b^2 \log n$; consequently, $G$ is $(b,c)$-loosely $(32b^2\log n)$-$\alpha$-separable.
	By Lemma \ref{lem:ta-bounded}, $\ta(G)\le 10c + 160b^2 \log n$.
	This completes the proof.
	
\end{proof}

\section{AI Disclosure}

We used GPT-6 Astra to proofread the manuscript.
No AI tools were used to create the proof.

\bibliographystyle{plain}
\bibliography{ref}

\end{document}